\documentclass[11pt]  {amsart} 
\usepackage{amsmath,latexsym,amssymb,amsmath,
amscd,amsthm,amsxtra}

\usepackage{color}

\numberwithin{equation}{section}

\newtheorem{theorem}{Theorem}[section]
\newtheorem{proposition}{Proposition}[section]
\newtheorem{lemma}{Lemma}[section]
\newtheorem{corollary}{Corollary}[section]

\newtheorem{remark}{\textbf{Remark}}[section]

\def\tr{\mathrm{tr}}
\def\o{\omega}
\def\p{\partial}

\def\a{\alpha}

\def\g{\gamma}
\def\d{\delta}
\def\l{\lambda}
\def\s{\sigma}

\def\n{\nabla}
\def\SS{{\mathbb S}}
\def\<{\langle}
\def\>{\rangle}

\def\ep{\epsilon}
\def\ep{\epsilon}
\def\beq{\begin{eqnarray}}
\def\eeq{\begin{eqnarray}}

\begin{document}

\title[{$L^p$ Christoffel-Minkowski problem:  the case $1< p<k+1$}]
{$L^p$ Christoffel-Minkowski problem: \\ the case $1< p<k+1$}
\author{Pengfei Guan and Chao Xia}

\address{Department of Mathematics and Statistics\\
McGill University\\
Montreal, H3A 0B9, Canada}
\email{pengfei.guan@mcgill.ca}

\address{School of Mathematical Sciences\\
Xiamen University\\
361005, Xiamen, P.R. China}
\email{chaoxia@xmu.edu.cn}
\thanks{Research of the first author is supported in part by an NSERC Discovery grant, the research of the second author is  supported in part by NSFC (Grant No. 11501480) and  the Natural Science Foundation of Fujian Province of China (Grant No. 2017J06003). Part of this work was done while CX was visiting the department of mathematics and statistics at McGill University. He would like to thank the department for its hospitality}

\date{}


\begin{abstract}{}
We consider a fully nonlinear partial differential equation associated to the intermediate $L^p$ Christoffel-Minkowski problem in the case $1<p<k+1$. We establish the existence of convex body with prescribed $k$-th even $p$-area measure on $\mathbb S^n$, under an appropriate assumption on the prescribed function. We construct examples to indicate certain geometric condition on the prescribed function is needed for the existence of smooth strictly convex body. We also obtain $C^{1,1}$ regularity estimates for admissible solutions of the equation when $ p\ge \frac{k+1}2$.

\end{abstract}



\medskip

\maketitle

\section{Introduction}

Convex geometry plays important role in the development of fully nonlinear partial differential equations. The classical Minkowski problem and the Christoffel-Minkowski problem in general, are beautiful examples of such interactions (e.g., \cite{Nirenberg, P52, CY, P78, B, F, GM}). The core of convex geometry is the Brunn-Minkowski theory, the Minkowski sum, the mixed volumes, curvature and area measures are fundamental concepts.
The notion of the Minkowski sum was extended by Firey \cite{Fir} , he introduced the so-called $p$-sum ($p>1$) for convex bodies. Lutwak \cite{Lut} further developed a corresponding Brunn-Minkowski-Firey theory based on Firey's $p$-sums. Lutwak initiated the study of the Minkowski problem for $p$-sums and established the uniqueness of the problem, along with the existence in the even case. The regularity of the solution in the even case was proved subsequently by Lutwak-Oliker \cite{LO}. Chou-Wang \cite{CW} and Guan-Lin \cite{GL} studied this problem from the PDE point of view, extensive study was carried out by Lutwak-Yang-Zhang in a series of papers, we refer \cite{LYZ, BLYZ} for further references in this direction.

This paper concerns the intermediate Christoffel-Minkowski problem related to $p$-sums, which we callit the $L^p$-Christoffel-Minkowski problem. While the $L^p$-Minkwoski problem corresponds to a Monge-Amp\`ere type equations, the $L^p$-Christoffel-Minkowski problem corresponds to a fully nonlinear partial differential equation of Hessian type. 

For a convex body $K$ in $\mathbb{R}^{n+1}$, we denote by $h(K, \cdot)$ its support function. For any $p\geq 1$, the $p$-sum of two convex bodies $K$ and $L$, $K+_p L$,  is defined through its support function,
\begin{eqnarray*}
h^p(\l_1 K+\l_2 L,\cdot)=\l_1 h^p(K, \cdot)+\l_2 h^p (L, \cdot), \quad \l_1, \l_2\in \mathbb{R}_+.
\end{eqnarray*}
The mixed $p$-quermassintegrals for $K$ and $L$ are defined by
$$W_{p,k}(K, L)=\lim_{\ep\to 0}\frac{W_k(K+_p \ep L)-W_k(K)}{\ep}, \quad p\ge1, 1\le k\le n.$$
Here $W_k(K)$ is the usual quermassintegral for $K$.
It was shown by Lutwak \cite{Lut} that $W_{p,k}(K,L)$ has the following integral representation:
\begin{eqnarray*}
W_{p,k}(K,L)=\frac{1}{n+1}\int_{\mathbb{S}^n} h(L, x)^ph(K, x)^{1-p} dS_{k}(K, x),
\end{eqnarray*}
where $dS_{k}(K,\cdot)$ is the $k$-th surface area measure of $K$. Thus $h(K, x)^{1-p} dS_{k}(K, x)$ is the local version of the  mixed $p$-quermassintegral. We call it  $k$-th $p$-area measure. When $p=1$, it reduces to the usual $k$-th area measures.

If $K$ is a convex body with $C^2$ boundary and support function $h$,  then
\begin{eqnarray*}
dS_{k}(K, \cdot)=\s_{n-k}(\n^2 h+h g_{\mathbb{S}^n})d\mu_{\mathbb{S}^n}.
\end{eqnarray*}
Here $\n^2 h$ is the Hessian on $\mathbb{S}^n$, $\s_{n-k}$ is the $(n-k)$-th elementary symmetric function. Therefore, to solve the Minkowski problem for $p$-sum is equivalent to solve the following PDE:\begin{eqnarray}\label{MAn}
\s_{n}(\n^2 u+u g_{\mathbb{S}^n})=u^{p_0}f \hbox{ on }\mathbb{S}^n,
\end{eqnarray}
where $p_0=p-1$.

After the development of $L^p$-Minkowski problem, it is natural to consider the $L^p$-Christoffel-Minkowski problem, i.e., the problem of prescribing the $k$-th $p$-area measure for general $1\leq k\leq n-1$ and $p\geq 1$.
As before, this problem can be reduced to the following nonlinear PDE:
 \begin{eqnarray}\label{CMpk}
\s_{k}(\n^2 u+u g_{\mathbb{S}^n})=u^{p_0}f\hbox{ on }\mathbb{S}^n.
\end{eqnarray}

A solution $u$ to \eqref{CMpk}  is called admissible if $(\n^2u+u g_{\mathbb{S}^n})\in \Gamma_k$ and $u$ is (strictly) spherically convex if $(\n^2u+u g_{\mathbb{S}^n})\ge 0$  $(>0)$.
For $k<n$ and $p_0=0$, the above is exactly the equation for the intermediate Christoffel-Minkowski problem of prescribing $k$-th area measures. Note that admissible solutions to equation (\ref{CMpk}) is not necessary a geometric solution to $L^p$-Christoffel-Minkowski problem if $k<n$. As in the classical Christoffel-Minkowski problem \cite{GM}, one needs to deal with the convexity of the solutions of (\ref{CMpk}).
Under a sufficient condition on the prescribed function, Guan-Ma \cite{GM} proved the existence of a unique convex solution. The key tool to handle the convexity is the constant rank theorem for fully nonlinear partial differential equations.
Equation (\ref{CMpk}) has been studied by Hu-Ma-Shen \cite{HMS} in the case $p_0\geq k$. In this case, there is a uniform lower bound for solutions if $f>0$ and they
proved the existence of convex solutions to \eqref{CMpk} under some appropriate sufficient condition.   The case $0<p_0<k$  is different, equation (\ref{CMpk}) is degenerate even for $f>0$ as there is no uniform lower bound for solutions in general.

The focus of this paper is to address two questions regarding equation (\ref{CMpk}) when $0<p_0<k$. \begin{enumerate}\item  When does there exist a {\it smooth convex} solution? \item Regularity of general {\it admissible solutions} of equation (\ref{CMpk}). \end{enumerate}

Our first result is the following.

\begin{theorem}\label{thm1}
Let $1\le k\le n-1$ be an integer and $0<p_0<k$ be a real number. For any positive {\em even}  function $f\in C^l(\SS^n)$ $(l\ge 2)$ satisfying
\begin{eqnarray}\label{convexity}
(\n^2 f^{-\frac{1}{k+p_0}}+f^{-\frac{1}{k+p_0}} g_{\mathbb{S}^n})\geq 0,
\end{eqnarray}
 there is a unique  {\em even, strictly spherically convex}  solution $u$ of the equation \eqref{CMpk}. Moreover, for each $\a\in (0,1)$, there is some constant $C$, depending on  $n, k, p_0, l,  \a, \min f$ and $\|f\|_{C^l(\SS^n)}$, such that
\begin{eqnarray}\label{esti}
\|u\|_{C^{l+1,\a}(\SS^n)}\leq C.
\end{eqnarray}
\end{theorem}

An immediate consequence of the previous theorem is the following existence result for the $L^p$ Christoffel-Minkowski problem for the case $1<p< k+1$.
\begin{corollary}\label{thm2}
Let $1\le k\le n-1$ be an integer and $1<p<k+1$ be a real number. For any positive {\em even} function $f\in C^l(\SS^n)$ $(l\ge 2)$ satisfying
\begin{eqnarray}\label{convexity0}
(\n^2 f^{-\frac{1}{k+p-1}}+f^{-\frac{1}{k+p-1}} g_{\mathbb{S}^n})\geq 0,
\end{eqnarray}
 there is a unique closed strictly convex hypersurface $M$ in $\mathbb{R}^{n+1}$ of class $C^{l+1,\a}$ (for all $0<\a<1$) such that the $(n-k)$-th $p$-area measure of $M$ is $fd\mu_{\SS^n}$.
 \end{corollary}

This is an analogue result of Lutwak-Oliker \cite{LO}. We use method of continuity to prove Theorem \ref{thm1}. The strictly convexity can be preserved along the continuity method by the constant rank theorem as in \cite{GM, HMS}. Unlike the case $p\ge k+1$, the lower bound of $u$ is not true in general if $p<k+1$. The crucial step is to show a uniform positive lower bound for $u$ under evenness assumption.  In contrast to the $L^p$-Minkowski problem \cite{LO}, the evenness assumption does not directly yield the lower bound of $u$ when $k<n$ as we do not have direct control of the volume of the associated convex body. The most technical part in this paper is to obtain a refined gradient estimate Proposition \ref{gradient} and to use it to prove Proposition \ref{lowerbound} with the assumption of evenness and spherical convexity of $u$.
One would like to ask that would condition (\ref{convexity}) guarantee the positivity of $u$?
We will exhibit some examples in section 5 to indicate that condition (\ref{convexity}) is not sufficient (see Proposition \ref{thm3}).

As in the case of the $L^p$-Minkowski problem \cite{GL}, one has $C^2$ estimate if $p_0\ge \frac{k-1}{2}$.

\begin{theorem}\label{thm4'}
Let $1\le k\le n-1$ be an integer and  $\frac{k-1}{2}\le p_0<k$  be a real number. 
For any positive function $f\in C^2(\SS^n)$ there exists a solution $u$ to \eqref{CMpk}  with $(\n^2 u+ug_{\SS^n})\in \bar \Gamma_k$. Moreover,
\begin{eqnarray*}
\|u\|_{C^{1,1}(\SS^n)}\leq C.
\end{eqnarray*}
where $C$ depends on $n, k, p_0, \|f\|_{C^2(\SS^n)}$ and $\min_{\SS^n} f$. Furthermore, solution is $C^2$ continuous (i.e., $\nabla^2 u$ is continuous) if $p_0>\frac{k-1}2$.
 \end{theorem}

From next section on, the range for $p_0$ is $0<p_0<k$ unless otherwise specified.


\medskip

\section{Preliminaries}

We recall the basic notations.

Let $\s_k(A)$ be the $k$-th elementary symmetric function defined on the set $\mathcal{M}_{n}$ of  symmetric $n\times n$ matrices  and
$\s_k(A_1, \cdots, A_k)$ be the complete polarization of $\s_k$ for $A_i\in \mathcal{M}_{n},  i=1,\cdots, k$, i.e.
\begin{eqnarray*}
\s_k(A_1, \cdots, A_k)=\frac{1}{k!}\sum_{\substack{i_1,\cdots i_{k}=1,\\ j_1,\cdots,j_{k}=1}}^n\delta_{j_1\cdots j_k}^{i_1\cdots i_k}A_{1_{i_1j_1}}\cdots A_{k_{i_kj_k}}.
\end{eqnarray*}
Let $\Gamma_k$ be  Garding's cone $$\Gamma_k=\{A\in\mathcal{M}_{n}: \s_i(A)>0 \hbox{ for }i=1,\cdots, k \}.$$

Let $(\SS^n, g_{\mathbb{S}^n})$ be the unit round $n$-sphere and $\nabla$ be the covariant derivative on $\SS^n$. For a  function $u\in C^2(\mathbb{S}^n)$, we denote  by $W_u$  the matrix $$W_u:=\n^2 u+u g_{\mathbb{S}^n}.$$ In the case $W_u$ is positive definite, the eigenvalue of $W_u$ represents the principal radii of a  strictly convex hypersurface with  support function $u$.

Let $u^i \in C^2(\mathbb{S}^n)$, $i=1,\cdots, n+1.$
Set
$$V(u^1, u^2, \cdots, u^{n+1}):=\int_{\SS^n} u^1\s_n( W_{u^2}, \cdots, W_{u^{n+1}})d\mu_{\SS^n}, $$
$$V_{k+1}(u^1, u^2, \cdots, u^{k+1}):=V(u^1, u^2, \cdots, u^{k+1}, 1, \cdots, 1).$$
We collect the following properties which have been proved in \cite{GMTZ}.

\begin{lemma}[\cite{GMTZ}]\label{lemma}

\

 {\rm (1)} $V_{k}(u^1, u^2, \cdots, u^{n+1})$ is a symmetric multilinear form on $(C^2(\SS^n))^{n+1}$.
In particular, \begin{eqnarray*}
V_{k+1}(\underbrace{u, \cdots, u}_{k+1})=V_{k+2}(1, \underbrace{u, \cdots, u}_{k+1}).
\end{eqnarray*}
Therefore, the  Minkowski's integral formula holds:
\begin{eqnarray}\label{Min}
\int_{\SS^n} u \s_k(W_u) d\mu_{\SS^n}=\frac{k+1}{n-k}\int_{\SS^n}  \s_{k+1}(W_u) d\mu_{\SS^n}.
\end{eqnarray}

{\rm (2)} Let $u^i\in C^2(\SS^n), i=1, 2, \cdots,k$ be such that $u^i>0$ and $W_{u^i}\in \Gamma_k$ for $i=1,2,\cdots, k$, Then for any $v\in C^2(\SS^n)$,the Alexandrov-Fenchel inequality holds:
\begin{eqnarray}\label{AF0}
V_{k+1}(v, u^1,\cdots, u^k)^2\geq V_{k+1}(v, v, u^2, \cdots, u^k)V_{k+1}(u^1, u^1, u^2\cdots, u^k),
\end{eqnarray}
the equality holds if and only if $v = au_1+\sum_{l=1}^{n+1} a_lx_l$ for some constants $a, a_1, \cdots, a_{n+1}$.
In particular, there are some sharp constant $C_{n,k}$ such that
\begin{eqnarray}\label{AF}
\left(\int_{\SS^n}  \s_{k+1}(W_u) d\mu_{\SS^n}\right)^{\frac{1}{k+1}}\le C_{n,k}\left(\int_{\SS^n}  \s_{k}(W_u) d\mu_{\SS^n}\right)^{\frac1k}.
\end{eqnarray}
\end{lemma}

\

Inequality \eqref{AF} in Lemma \ref{lemma}  follows from Alexandrov-Fenchel's  inequality \eqref{AF0} and Minkowski's  formula \eqref{Min} via an iteration argument.

We remark that in Lemma \ref{lemma} (2), it is sufficient to assume   $W_{u^i}\in \Gamma_k$ instead that $W_u$ is positive definite which is the classical assumption from convex geometry.

\medskip

We list some other known results which will be used in the rest of the sections.

The following theorem was proved for \eqref{CMpk} by Hu-Ma-Shen in \cite{HMS}, a generalization of the constant rank theorem in \cite{GM}.

\begin{theorem}[\cite{HMS}]\label{const rank}
Let $p_0>0$.
Let $u$ be a positive solution to \eqref{CMpk} such that $W_u$ is positive semi-definite. Then if $f^{-\frac{1}{p_0+k}}$ is spherically convex,
then $W_u$ is positive definite.
\end{theorem}

\medskip

The following lemma is a special case of Lemma 1 in \cite{GLL}, we state it for $W_u\in C^{1}(\mathbb S^n)$.
\begin{lemma}\label{gll-lemma} Let $e_1, \cdots, e_n$ be a local orthonormal frame on $\mathbb{S}^n$, denote $\nabla_s=\nabla_{e_s}, \forall s=1,\cdots, n$, then $\forall W=W_u\in \Gamma_k\cap C^{1}(\mathbb S^n)$, $k\ge 2$,
\begin{eqnarray}\label{concavity}
&&-\sigma_k^{ij,lm}\nabla_s W_{ij}\nabla_s W_{lm}\geq \s_k\left[\frac{\n_s \s_k}{\s_k}-\frac{\n_s \s_1}{\s_1}\right]\left[(\frac{1}{k-1}-1)\frac{\n_s \s_k}{\s_k}-(\frac{1}{k-1}+1)\frac{\n_s \s_1}{\s_1}\right].
\end{eqnarray} \end{lemma}

\medskip

\section{A priori estimate for admissible solutions}\label{admissible}

In this section we establish $C^1$ a priori estimates for the admissible solutions of \eqref{CMpk}.

\subsection{A gradient estimate}
\begin{proposition}\label{gradient}
Let $u$ be a positive admissible solution to \eqref{CMpk}. Set $m_u=\min u$ and $M_u=\max u$.  Then there exist some positive constants $A$ and $0<\g<1$, depending on $n, k, \min f$ and $\|f\|_{C^1}$,  such that $$\frac{|\n u|^2}{|u-m_u|^\g}\leq A M_u^{2-\gamma}.$$
\end{proposition}
\begin{proof}
Let $\Phi=\frac{|\n u|^2}{(u-m_u)^\g}$, where $0<\g<1$ is to be determined. First we claim $\Phi$ is well-defined, in other words, $\Phi$ can be defined at the minimum points. Consider $\Phi_\epsilon=\frac{|\n u|^2}{(u-m_u+\epsilon)^\g}$ for $\epsilon>0$. Then at a maximum point of $\Phi_\epsilon$, we have $$(\n^2 u+u I) \n u=(\frac{\g}{2} \frac{|\n u|^2}{u-m_u+\epsilon}+u)\n u.$$ Hence
$$\Phi_\epsilon\leq \frac{2}{\g}(u-m_u+\epsilon)^{1-\g}\max_{\SS^n} \l_{max}(W_u),$$ where $\l_{max}(W_u)$ is the largest eigenvalue of $W_u$.
Thus when $\g<1$, we have $\Phi_\epsilon(y)\to 0$ for $u(y)=m_u$ as $\epsilon\to 0$. Therefore, it make sense to define $\Phi=0$ at the minimum point of $u$.

Assume $\Phi$ attains its maximum at $x_0$. Then $u(x_0)>m_u$. By using the orthonormal frame and rotating the coordinate, we can assume $g_{ij}(x_0)=\d_{ij}$, $u_1(x_0)=|\n u|(x_0)$ and  $u_i(x_0)=0$ for $i=2,\cdots, n$. In the following we compute at $x_0$. By the critical condition,
\begin{eqnarray*}
\frac{2u_lu_{li}}{|\n u|^2}=\g\frac{u_i}{u-m_u}\hbox{ for each }i.
\end{eqnarray*}
Thus $u_{1i}=0$ for $i=2,\cdots, n$ and
\begin{eqnarray}\label{critical}
u_{11}=\frac{\g}{2}\frac{u_1^2}{u-m_u}.
\end{eqnarray}
By rotating the remaining $n-1$ coordinates, we can assume $(u_{ij})$ is diagonal. Consequently, $F^{ij}=\frac{\p \sigma_k}{\p W_{ij}}$ is also diagonal.

We may assume $\Phi\geq AM_u^{2-\gamma}$, where $A$ is a large constant to be determined. Then
\begin{eqnarray}\label{u11large}
u_{11}=\frac{\g}{2}\frac{u_1^2}{u-m_u}\geq \frac{\g}{2}\frac{AM_u^{2-\gamma}}{(u-m_u)^{1-\g}}\geq  \frac{\g}{2}AM_u.
\end{eqnarray}
Since $u\leq M_u$, for $\delta>0$, we may choose $A$ with $A>> \frac{2}{\g}$  such that \begin{eqnarray}\label{wu}
W_{11}=u_{11}+u\leq (1+\delta)u_{11}.
\end{eqnarray}

As $W_{ii}\ge u_{ii}$, by the maximal condition and \eqref{critical},
\begin{eqnarray}\label{11}
0&\geq &F^{ii}(\log\Phi)_{ii}\\&=& F^{ii}\frac{2u_{ii}^2+2u_lu_{lii}}{|\n u|^2}-\g \frac{F^{ii}u_{ii}}{u-m_u}+\g(1-\g)\frac{F^{ii}u_i^2}{(u-m_u)^2}\nonumber\\&=& \frac{2F^{ii}u_{ii}^2}{u_1^2}+\frac{2F^{ii}u_1(W_{ii1}-u_i\delta_{1i})}{u_1^2}-\g \frac{F^{ii}u_{ii}}{u-m_u}+\g(1-\g)\frac{F^{ii}u_i^2}{(u-m_u)^2}
\nonumber\\&=& \frac{2F^{ii}u_{ii}^2}{u_1^2}+2p_0u^{p_0-1}f+\frac{ 2u^{p_0} f_1}{u_1}-2F^{11}-\g \frac{F^{ii}u_{ii}}{u-m_u}+\g(1-\g)\frac{F^{ii}u_i^2}{(u-m_u)^2}\nonumber\\&\geq &\frac{2F^{ii}u_{ii}^2}{u_1^2}+ \g(1-\g)\frac{F^{11}u_1^2}{(u-m_u)^2}+\frac{ 2u^{p_0} f_1}{u_1} -2F^{11}-\g  \frac{F^{ii}W_{ii}}{u-m_u}\nonumber\\
&= &\frac{2F^{ii}u_{ii}^2}{u_1^2}+2(1-\g)\frac{F^{11}u_{11}}{u-m_u}+\frac{ 2u^{p_0} f_1}{u_1} -2F^{11}-k\g  \frac{\sigma_k(W)}{u-m_u}\nonumber\\
&= &\sum_{i\neq 1}\frac{2F^{ii}u_{ii}^2}{u_1^2}+ 2(1-\g)\frac{F^{11}u_{11}}{u-m_u}+\frac{ 2u^{p_0} f_1}{u_1} +2F^{11}(\frac{u_{11}^2}{u_1^2}-1) -k\g  \frac{\sigma_k(W)}{u-m_u}\nonumber.\nonumber
\end{eqnarray}
By (\ref{critical}) and (\ref{u11large}), if $A\ge \frac{4}{\g^2}$,
\begin{eqnarray}\label{00}
\frac{u_{11}^2}{u_1^2}-1\ge \frac{\g^2}4 A\frac{M_u}{u-m_u}-1\ge 0.
\end{eqnarray}

Using the definition of $\Phi$, we have
\begin{eqnarray}\label{77}
-\frac{ 2u^{p_0} f_1}{u_1}&\geq &-Cu^{p_0}\Phi^{-\frac12}(u-m_u)^{-\frac{\g}{2}} \nonumber\\
&\geq &-\frac{C}{\sqrt{A}}M^{-1+\frac{\g}{2}}u^{p_0}(u-m_u)^{-\frac{\g}{2}}\nonumber\\
&\geq& -\frac{C}{\sqrt{A}}u^{p_0}(u-m_u)^{-1}\nonumber\\
&\geq& -\frac{C}{\sqrt{A}}\frac{\sigma_k(W)}{u-m_u}\nonumber.
\end{eqnarray}

For $N>1$ to be determined later, denote \begin{eqnarray*}
K= \{i:  u_{ii}> Nu_{11}\}.
\end{eqnarray*}
When $A$ is large enough, by (\ref{wu}), \[u_{ii}=W_{ii}-u=W_{ii}-\delta u_{11}\ge W_{ii}-\delta u_{ii}, \forall i\in K.\] Hence
\begin{eqnarray}\label{100}
\sum_{i\in K}\frac{2F^{ii}u_{ii}^2}{u_1^2}\geq \sum_{i\in K}\frac{2NF^{ii}u_{ii}u_{11}}{u_1^2}=N\g \sum_{i\in K}\frac{F^{ii}u_{ii}}{u-m_u}\geq \frac{N\g}{1+\delta} \sum_{i\in K}\frac{F^{ii}W_{ii}}{u-m_u}.
\end{eqnarray}
Combining (\ref{critical})-(\ref{100})
\begin{eqnarray}\label{new11}
\quad  0\geq  \frac{N\g}{1+\delta} \sum_{i\in K}\frac{F^{ii}W_{ii}}{u-m_u}+ \frac{ 2(1-\g)}{1+\delta}\frac{F^{11}W_{11}}{u-m_u}-\frac{C}{\sqrt{A}}\frac{\sigma_k(W)}{u-m_u}-k\g  \frac{\sigma_k(W)}{u-m_u}.\end{eqnarray}

Let's denote $W_{mm}=\max\{W_{ii}| i=1,\cdots, n\}$.
We have
\[\sigma_k(W)=\sigma_{k-1}(W | m)W_{mm}+ \sigma_{k}(W | m).\]
If $\sigma_k(W|m)\le 0$, then \[\sigma_{k-1}(W | m)W_{mm}\ge \sigma_k(W).\]
Let's assume $\sigma_k(W|m)> 0$, that implies $(W|m)\in \Gamma_k$. In turn,  $\s_{k-1}(W|mi)>0, \forall i\neq m$ and
\[k\s_k(W|m)=\sum_{i\neq m} W_{ii}\s_{k-1}(W|mi)\le W_{mm}\sum_{i\neq m}\s_{k-1}(W|mi)= (n-k)W_{mm}\s_{k-1}(W|m).\]

Combining the above inequalities, we have
\begin{equation}\label{neweq01}  \sigma_{k-1}(W | m)W_{mm}\ge \frac{k}{n}\sigma_k(W).\end{equation}
If $K\neq \emptyset$, then $m\in K$, and
\begin{equation}\label{neweq02} \sum_{i\in K}\frac{F^{ii}W_{ii}}{u-m_u}\ge  \frac{F^{mm}W_{mm}}{u-m_u}\ge \frac{k}{n} \frac{\sigma_k(W)}{u-m_u}.\end{equation}
If $K= \emptyset$, then $0\le W_{11}\le W_{mm}\le N W_{11}$, as $F^{11}\ge F^{mm}$,
\begin{equation}\label{neweq03}
\frac{F^{11}W_{11}}{u-m_u}\ge \frac{1}{N}\frac{F^{mm}W_{mm}}{u-m_u}\ge \frac{k}{Nn}\frac{\sigma_k(W)}{u-m_u}.\end{equation}

Combining \eqref{new11}, \eqref{neweq02} and \eqref{neweq03},
\begin{eqnarray}\label{110}
0 &\geq & \big[\min\{\frac{Nk\g}{n(1+\delta)}, \frac{ 2k(1-\g)}{Nn(1+\delta)}\} -\frac{C}{\sqrt{A}}-k\g\big]\frac{\sigma_k(W)}{u-m_u}>0,
\end{eqnarray} if we pick $N=n(1+2\delta)$, $\g=\frac{2}{N^2+2}$, and $A$ sufficiently large (for any $\delta>0$ fixed, e.g, $\delta=\frac{1}{10^{10}}$). This is a contradiction. Thus for our choice of $\g$ and $A$, we must have $\Phi\leq AM_u^{2-\g}$ at its maximum.
\end{proof}
When $k=n$, similar result was proved in \cite{Lu} where upper bound of $u$ was readily available. We note the proof of Proposition \ref{gradient} also works for certain range of $p_0<0$.

\subsection{Upper bound of $u$}

We now use raw $C^1$ estimate in Proposition \ref{gradient} to get an upper bound of $u$.
\begin{proposition}\label{C0}
Let $u$ be a positive admissible solution to \eqref{CMpk}. Then there exist some positive constants $c_0$ and $C_0$, depending on $n, k, p_0, \min f$ and $\int_{\SS^n}f$, such that $$0<c_0\leq \max u\leq C_0.$$
\end{proposition}
\begin{proof}
Let $x_0$ be a maximum point of $u$. Then
$\n^2 u(x_0)\le 0$.
It follows that
\begin{eqnarray*}
\binom{n}{k}u^k(x_0)\geq \sigma_k(W_u)(x_0)=u^{p_0}(x_0) f(x_0),
\end{eqnarray*}
and in turn we have
\begin{eqnarray}\label{maxlb}
\max_{\SS^n} u=u(x_0)\geq \left(\frac{\min_{\SS^n} f}{\binom{n}{k}}\right)^{\frac{1}{k-p_0}}.
\end{eqnarray}
From Proposition \ref{gradient}, we know $|\n u|^2(x)\leq AM_u^2=Au(x_0)^2$ for  any $x\in \SS^n$,
we have
\begin{eqnarray}\label{harnack}
u(x)\geq \frac12 u(x_0) \hbox{ if } dist(x,x_0)\leq \frac{1}{2\sqrt{A}}.
\end{eqnarray}
Thus
\begin{eqnarray}\label{ieq1}
&&\int_{\SS^n} u^{p_0+1}f\geq \int_{\{x\in\SS^n: dist(x,x_0)\leq \frac{1}{2\sqrt{A}}\}} u^{p_0+1}f\\&\geq &\frac{1}{2^{p_0+1}}u^{p_0+1}(x_0)\min_{\SS^n} f \left|\{x\in\SS^n: dist(x,x_0)\leq\frac{1}{2\sqrt{A}}\}\right|.\nonumber
\end{eqnarray}
On the other hand, using Minkowski's integral formula \eqref{Min}, Alexandrov-Fenchel's inequality \eqref{AF}, H\"older's inequality and \eqref{CMpk}, we have
\begin{eqnarray}\label{ieq11}
&&\int_{\SS^n} u^{p_0+1}f=\int_{\SS^n} u\sigma_k(W_u)=\frac{k+1}{n-k}\int_{\SS^n} \s_{k+1}(W_u)\\&\leq& C\left(\int_{\SS^n} \s_{k}(W_u)\right)^{\frac{k+1}{k}}=C\left(\int_{\SS^n} u^{p_0} f\right)^{\frac{k+1}{k}}\nonumber \\&\leq &C\left(\int_{\SS^n} u^{p_0+1}f\right)^{\frac{p_0}{p_0+1}\frac{k+1}{k}}\left(\int_{\SS^n} f\right)^{\frac{1}{p_0+1}\frac{k+1}{k}}.\nonumber
\end{eqnarray}
Since $p_0<k$, it follows from \eqref{ieq11} that
\begin{eqnarray}\label{ieq2}
&&\int_{\SS^n} u^{p_0+1}f\leq C\left(\int_{\SS^n} f\right)^{\frac{k+1}{k-p_0}}.
\end{eqnarray}
Combining \eqref{ieq1} and \eqref{ieq2}, we obtain $u\leq u(x_0)\leq C$.
\end{proof}

Combining Proposition \ref{gradient} and Proposition \ref{C0} we get full $C^1$ estimate.
\begin{proposition}\label{gradient1}
Let $u$ be an admissible solution to \eqref{CMpk}. Set $m_u=\min u$.  Then there exist some positive constants $0<\g<1$ and $C$, depending on $n, k, p_0, \min f$ and $\|f\|_{C^1}$,  such that $$\frac{|\n u|^2}{|u-m_u|^\g}\leq C.$$
\end{proposition}

\medskip

\section{Convex solutions}

So far, we have been dealing with general admissible solutions of equation (\ref{CMpk}). In order to solve the $L^p$-Christoffel-Minkowski problem, we need to establish the existence of convex solutions, i.e., solutions to (\ref{CMpk}) with $W_u\ge 0$. As in the case of the classical Christoffel-Minkowski problem \cite{GM}, one needs some sufficient conditions on the prescribed function $f$ in equation (\ref{CMpk}) when $k<n$. Unlike the classical Christoffel-Minkowski problem, equation (\ref{CMpk}) may degenerate when $p_0>0$ in general.  We first derive lower bound of convex solutions.

\subsection{Lower bound for $u$}
\

To get a uniform positive lower bound, we need to impose {\it evenness} assumption together with $W_u\ge 0$. We remark that such estimate was straightforward when $k=n$ since the equation implies a positive lower bound of volume. For $k<n$, a lower bound on quermassintegral $V_{k+1}$ does not guarantee the non-degeneracy of the convex body. We need some extra effort.

\begin{proposition}\label{lowerbound}
Let $u$ be a positive, even, spherically convex solution to \eqref{CMpk}. Then there exists some positive constant $C$, such that $$u\ge C>0.$$
\end{proposition}

\begin{proof} Since $u$ is even, we can assume without loss of generality that $u(x_1)=\max u=: M$ and $u(x_2)=\min u$ and $dist(x_1,x_2)=: 2d\leq \frac{\pi}{2}$.  So $d\leq \frac{\pi}{4}$. Let $\g: [-d,d]\to \SS^n$ be the arc-length parametrized geodesic such that $\g(-d)=x_1$ and $\g(d)=x_2$. Let $u: [-d, d]\to \mathbb{R}$ be the function $u(t)=u(\g(t))$ and denote \begin{eqnarray}\label{ode}
u''(t)+u(t)=g(t).
\end{eqnarray}
It follows from the critical condition of $u$ at $x_1$ and $x_2$ that \begin{eqnarray}\label{bvp}
u'(-d)=u'(d)=0, \quad u(-d)=M.
\end{eqnarray}

Let us explore the boundary value problem for the ODE, \eqref{ode} and \eqref{bvp}.
 It is easy to check that  $A\cos t+B\sin t$ is the general solutions to homogeneuous ODE $$u''+u=0$$ and a special solution to \eqref{ode} is given by $\cos t\int_{-d}^t \frac{1}{\cos^2 \tau}\int_{-d}^\tau g(s)\cos s ds d\tau$. Combining with the boundary condition \eqref{bvp}, we see all the solutions to \eqref{ode} and \eqref{bvp} are
 \begin{eqnarray}\label{odesolution}
u(t)=\cos t\int_{-d}^t \frac{1}{\cos^2 \tau}\int_{-d}^\tau g(s)\cos s ds d\tau+M\cos (d+t).
\end{eqnarray}
For simplicity, we denote by $G(\tau)=\int_{-d}^\tau g(s)\cos s ds$.
It follows from \eqref{odesolution}  that
\begin{eqnarray}
u(d)=\cos d \int_{-d}^d \frac{G(\tau)}{\cos^2 \tau}d\tau+M\cos 2d.
\end{eqnarray}
Our aim is to derive a positive lower bound of $\min u=u(d).$ In the case $2d\leq \frac{1}{2\sqrt{A}}$, from \eqref{harnack} and \eqref{maxlb},  $$\min u=u(x_2)\geq \frac12 u(x_1)=\frac12 M\geq \frac12\left(\frac{\min_{\SS^n} f}{\binom{n}{k}}\right)^{\frac{1}{k-p_0}}.$$ We consider now the case $2d\geq  \frac{1}{2\sqrt{A}}$.
From the definition of $G(\tau)$, by performing integration by parts, we have
\begin{eqnarray}\label{Gtau}
G(\tau)&=& \int_{-d}^\tau g(s)\cos s ds \nonumber\\
&=& \int_{-d}^\tau (u^{''}(s)+u(s))\cos s ds\nonumber\\
&=& u'(s)\cos s|_{-d}^\tau +\int_{-d}^\tau u'(s)\sin s+u(s)\cos s ds\nonumber\\
&=& u'(\tau)\cos \tau+ u(s)\sin s|_{-d}^\tau \nonumber\\
&=& u'(\tau)\cos \tau+ u(\tau)\sin \tau-M\sin(-d),\end{eqnarray}
where facts $u{'}(-d)=0, u(-d)=M$ are used.
In particular, as $u{'}(d)=0$,
\[G(d)=(u(d)+M)\sin d.\]

Since $\sin d\ge \sin(\frac{1}{4\sqrt{A}})$ and $u\ge 0$,  we see
\begin{eqnarray}\label{neweq1}
G(d)\ge M\sin(\frac{1}{4\sqrt{A}})>0.
\end{eqnarray}



By Proposition \ref{gradient1}, such that for $\tau\in [-d, d]$, $$|u'(\tau)|\leq C(u(\tau)-u(d))^{\frac{\g}2}\leq C\max_{\SS^n} |\n u|^{\frac{\g}2} |\tau-d|^{\frac{\g}2} \leq \tilde C|\tau-d|^{\frac{\g}2}.$$
Therefore, $G(\tau)$ is continuous as a function of $\tau$ from (\ref{Gtau}), and
\begin{eqnarray}\label{neweq2}
G(\tau)\geq G(d)-C^{*}|\tau-d|^{\frac{\g}2}, \quad \forall \tau\in [-d, d].
\end{eqnarray}
Note that $g(t)\ge 0$ as $W_u\ge 0$, hence $G(\tau)\ge 0, \forall \tau\in [-d,d]$. Take $\delta=(\frac{G(d)}{2C^{*}})^{\frac2{\g}}$, it follows from \eqref{odesolution}, \eqref{neweq1}, \eqref{neweq2} and $d\in [0,\frac{\pi}{4}]$ that
\begin{eqnarray*}
u(d)&=&\cos d \int_{-d}^d \frac{G(\tau)}{\cos^2 \tau}d\tau+M\cos 2d\geq
\cos d \int_{d-\delta}^d  G(\tau) d\tau\nonumber\\
&\geq &\frac{\sqrt{2}}{2}\cdot\frac12G(d)\delta\geq \frac{\sqrt{2}}{4}(2C^{*})^{-\frac 2 \g}\left(M\sin(\frac1{4\sqrt{A}})\right)^{1+\frac2{\g}}.
\end{eqnarray*}

\end{proof}

\subsection{Higher regularity}
\

\begin{proposition}\label{estimate}
Let $u$ be a positive, even, spherically convex solution to \eqref{CMpk}. For any $l\in \mathbb{R}$ and $0<\a<1$, there exists some positive constant $C$, depending on $n, k, p_0, l, \min f$ and $\|f\|_{C^l}$, such that \eqref{esti} holds.
\end{proposition}
\noindent{\it Proof of Proposition \ref{estimate}.}
From Proposition \ref{C0} and \ref{lowerbound}, we see $u$ is bounded from above and below by uniform positive constants. When $k=1$, as we already have $C^1$ bounds for $u$, higher regularity follows from elliptic linear PDE. We may assume $k\ge 2$.
Let $$\tilde{F}(W_u):=\s_k^{\frac1k}(W_u)=  (u^{p_0} f)^{\frac1k}.$$
Differentiating the equation  twice, we have
\begin{eqnarray}\label{equu}
\Delta (u^p f)^{\frac1k}&=&\tilde{F}^{ii}W_{iiss}+\tilde{F}^{ij,lm}W_{ijs}W_{klm}\nonumber\\&=&
\tilde{F}^{ii}(W_{ssii}-W_{ss}+nW_{ii})+\tilde{F}^{ij,lm}W_{ijs}W_{klm}\nonumber
\\&\le &\tilde{F}^{ii}(\s_1)_{ii}-\sum_i \tilde{F}^{ii}\s_1+n\s_k
\\&=& \tilde{F}^{ii}(\s_1)_{ii}-\sum_i \tilde{F}^{ii}\s_1+n u^{p_0} f.
\end{eqnarray}
where we used the concavity of $\tilde{F}$.

Note that $|\Delta (u^p f)^{\frac1k}|\le C\sigma_1$ and $\sum_i \tilde{F}^{ii}=\frac{n}{k}\s_k^{\frac{1-k}{k}}\s_{k-1}\ge  C_{n,k}\sigma_k^{-\frac1{k(k-1)}}\sigma_1^{\frac1{k-1}}$. Applying the maximum principle on \eqref{equu}, we see that $\s_1\le C$. Thus $\|u\|_{C^2}\le C$.
Since $W_u\ge 0$, we see that the equation is uniformly elliptic. Our assertion follows now from the standard Evans-Krylov and Schauder estimates.
\qed

\begin{remark}\label{rem}
The conditions that $u$ is even and $W_u\ge 0$ have been only used  in Proposition \ref{lowerbound}.
\end{remark}


\subsection{Existence}
\



In the following we use the continuity method to prove the existence and uniqueness of  strictly convex solutions.

\noindent{\bf Proof of Theorem \ref{thm1}.}

We first show that the solution is unique. The uniqueness of strictly spherically convex solution was showed by Lutwak \cite{Lut}. For convenience of readers, we give a proof on the uniqueness of admissible solutions.

Assume $u, v$ are two admissible solutions to \eqref{CMpk}. 
Then we have \begin{eqnarray}\label{eq'}
&&\s_k(W_u)=u^{p_0} f,\quad \s_k( W_v)=v^{p_0} f.\end{eqnarray}
Multiplying $v$ to the first equation in \eqref{eq'} and integrating over $\SS^n$, we have by using the Alexandrov-Fenchel inequality \eqref{AF0}
\begin{eqnarray}\label{uv1}
&&\int_{\SS^n} vu^{p_0}  f=\int_{\SS^n} v\s_k(W_u)=V_{k+1}(v, u,\cdots, u) \nonumber\\&\geq &V_{k+1}(u,u,\cdots,u)^{\frac{k}{k+1}} V_{k+1}(v, v,\cdots,v) ^{\frac{1}{k+1}}\nonumber
\\&=&\left(\int_{\SS^n} u^{p_0+1}f\right)^{\frac{k}{k+1}} \left(\int_{\SS^n} v^{p_0+1}f\right)^{\frac{1}{k+1}}.\end{eqnarray}
On the other hand, using H\"older's inequality,
\begin{eqnarray}\label{uv2}
\int_{\SS^n} vu^{p_0} f\leq\left(\int_{\SS^n} v^{p_0+1}f\right)^{\frac{1}{{p_0}+1}} \left(\int_{\SS^n} u^{p_0+1}f\right)^{\frac{{p_0}}{{p_0}+1}}.
\end{eqnarray}
Combining  \eqref{uv1} and \eqref{uv2}, in view of $0<p_0<k$, we obtain
\begin{eqnarray*}
\int_{\SS^n} u^{p_0+1}f\leq \int_{\SS^n} v^{p_0+1}f.
\end{eqnarray*}
Similar argument by interchanging the role of $u$ and $v$ gives
\begin{eqnarray*}
\int_{\SS^n} v^{p_0+1}f\leq \int_{\SS^n} u^{p_0+1}f.
\end{eqnarray*}
Thus all the above inequalities are equalities. In particular, equality holds in (\ref{uv1}). That is, equality holds in (\ref{AF}). In view of (\ref{eq'}), we must have $u\equiv v.$

We now prove the existence. Denote
$$f_t=\left(tf^{-\frac{1}{p_0+k}}+(1-t)\binom{n}{k}^{-\frac{1}{p_0+k}}\right)^{-(p_0+k)} \hbox{ for }  t\in [0, 1].$$ Then $f_t$ is even and satisfies \eqref{convexity}. Consider the equation
\begin{eqnarray}\label{conti-eq}
\s_k(\n^2 u+ug_{\SS^n})=u^{p_0}f_t.
\end{eqnarray}

Let $$S=\{t\in [0,1]| \eqref{conti-eq} \hbox{ has a positive, even solution $u_t$ with $W_{u_t}> 0$}\}.$$ It is clear that $u_0\equiv 1$ is a positive, even solution of \eqref{conti-eq} with $W_{u_0}> 0$ for $t=0$. Thus $S$ is non-empty.

Next we show $S$ is open. 
 The linearized operator at $u$ is given by  $$L_u(v):=\s_k^{ij}(W_u)(W_v)_{ij}-p_0u^{p_0-1}vf_t=\s_k^{ij}(W_u)(W_v)_{ij}-p_0 u^{-1}v\s_k(W_u).$$ Suppose $L_u(v)=0$.
Then \begin{eqnarray}\label{lin1}
\s_k^{ij}(W_u)(W_v)_{ij}-p_0u^{-1}v\s_k(W_u)=0.
\end{eqnarray}
Multiplying \eqref{lin1} with $u$ and integrating over $\SS^n$, we have
\begin{eqnarray*}
k\int_{\SS^n} v\s_k(W_u)=\int_{\SS^n} u\s_k^{ij}(W_u)(W_v)_{ij}=p_0\int_{\SS^n} v\s_k(W_u)
\end{eqnarray*}
Since $k\neq p_0$, we have $\int_{\SS^n} v\s_k(W_u)=0$.
On the other hand, Multiplying \eqref{lin1} with $v$ and integrating over $\SS^n$, we have
\begin{eqnarray*}
kV(v,v, u, \cdots, u)=\int_{\SS^n} v\s_k^{ij}(W_u)(W_v)_{ij}=p_0\int_{\SS^n} u^{-1}v^2\s_k(W_u)
\end{eqnarray*}

Since $V(v, u, \cdots, u)=\int_{\SS^n} v\s_k(W_u)=0$, by using  the Alexandrov-Fenchel inequality \eqref{AF0}, we see $$V(v,v, u, \cdots, u)\leq 0.$$
Thanks to $p_0>0$, we have $\int_{\SS^n} u^{-1}v^2\s_k(W_u)
\leq 0$, which implies $v\equiv 0$.
Hence the kernel of the linearized operator of the equation is trivial.
By the implicit function theorem, for each $t_0\in S$, there exists a neighborhood $\mathcal{N}$ of $t_0$ such that there exists a positive solution $u_t$ of \eqref{conti-eq} with $W_{u_t}> 0$ for $t\in \mathcal{N}$. Since $f_t$ is even, it follows from the uniqueness result that $u_t$ must be even. Hence, $\mathcal{N}\subset S$ ans $S$ is open.

We now prove the closeness of $S$.
Let $\{t_i\}_{i=1}^\infty\subset S$ be a sequence such that $t_i\to t_0$ and ${u_{t_i}}$ be a positive even solution to \eqref{conti-eq} with $W_{u_{t_i}}>0$ for  $t=t_i$. By virtue of the a priori estimate in Theorem \ref{estimate}, there exists a subsequence, still denote by  $u_{t_i}$, converges to some function $u$ in $C^{l+1}$ norm. In particular, $u$ is an even solution to \eqref{conti-eq}  for $t=t_0$. Suppose $W_u$ is not positive definite, then $W_u$ is positive semi-definite and $\det(W_u)(x_0)=0$ for some $x_0\in \SS^n$. Since $f_{t_0}$ satisfies \eqref{convexity}, we know from the constant rank theorem that $W_u$ must be positive definite. A contradiction.
Therefore, $t_0\in S$ and $S$ is closed.

 We conclude that $S=[0,1]$ and \eqref{conti-eq} with $t=1$, which is \eqref{CMpk}, has a positive even solution $u$ with $W_u>0$. The proof is completed. \qed

\medskip

 \section{Examples}

For the Minkowski problem for $p$-sum with $1< p< n+1$, a $C^2$ convex hypersurface to the $L^p$ Minkowski problem does not always exist even if $f$ is a smooth positive function. A series of counterexamples have been constructed in \cite{GL}. The arguments in \cite{GL} can be extended to construct similar examples for equation (\ref{CMpk}).

Let $\a=\frac{k}{k-p_0}$. Set $u(x)=(1-x_{n+1})^\a$, where $x= (x_1,\cdots, x_n, x_{n+1})=:(x', x_{n+1})\in \SS^n\subset \mathbb{R}^{n+1}.$
We view the open hemisphere $\SS^n_+$, centered at the north pole, as a graph over $\{x'\in \mathbb{R}^n: |x'|^2<1\}$.
The metric $g$ and its inverse $g^{-1}$  on $\SS^n_+$ are $$g_{ij}=\d_{ij}+\frac{x_ix_j}{1-|x'|^2},\quad g^{ij}=\d_{ij}-x_ix_j$$ and the Christoffel symbol is $$\Gamma_{ij}^l=g_{ij}x_l.$$
In this local coordinates, $u(x)=(1-\sqrt{1-|x'|^2})^\a$. By a direct computation,  we have
\begin{eqnarray*}
&&g^{il}\n^2_{jl} u+u\delta_{ij}=g^{il}(\p_i\p_j u-\Gamma_{il}^m\p_m u)+u\delta_{ij}
\\&=&(1-\sqrt{1-|x'|^2})^{\a-1}\left((\a-1) \sqrt{1-|x'|^2}+1\right)\d_{ij}+\a(\a-1)(1-\sqrt{1-|x'|^2})^{\a-2}x_ix_j
\end{eqnarray*}
Using $(\a-1)k=\a p_0$, we see
\begin{eqnarray*}
&&\sigma_k(g^{il}\n^2_{jl} u+u\delta_{ij})=u^{p_0} f
\end{eqnarray*}
where \begin{eqnarray*}
&&f=\sigma_k\left[\left((\a-1) \sqrt{1-|x'|^2}+1\right)\d_{ij}+\a(\a-1)\frac{x_ix_j}{1-\sqrt{1-|x'|^2}}\right].
\end{eqnarray*}
It is clear that the eigenvalues of  the matrix $(\d_{ij}+b\frac{x_ix_j}{|x|^2})$ are $1$ with multiplicity $n-1$ and $1+b$ with multiplicity $1$.  Thus
\begin{eqnarray*}
f&=&\left((\a-1) \sqrt{1-|x'|^2}+1\right)^k\times\\&&\left[\binom{n-1}{k}+\binom{n-1}{k-1}\left(1+\frac{\a(\a-1)|x'|^2}{(1-\sqrt{1-|x'|^2})\left((\a-1) \sqrt{1-|x'|^2}+1\right)}\right)\right]
\\&=&\binom{n}{k}\left((\a-1) \sqrt{1-|x'|^2}+1\right)^k\\&&+\binom{n-1}{k-1}\frac{\a(\a-1)|x'|^2}{1-\sqrt{1-|x'|^2}}\left((\a-1) \sqrt{1-|x'|^2}+1\right)^{k-1}.
\end{eqnarray*}
Since $$\sqrt{1-|x'|^2}=1-\frac12|x'|^2-\frac18|x'|^4+o(|x'|^4),$$ we have
\begin{eqnarray*}
f&=&\binom{n}{k}\left(\a^k -\frac12 k\a^{k-1}(\a-1)|x'|^2\right)\\&&+2\a(\a-1)\binom{n-1}{k-1}(1-\frac14|x'|^2)\left[\a^{k-1}-\frac12(k-1)\a^{k-2}(\a-1)|x'|^2\right]+o(|x'|^2)
\\&=&\left[\binom{n}{k}+2(\a-1)\binom{n-1}{k-1}\right]\a^k\\&&-\frac12\a^{k-1}(\a-1)\left[ k\binom{n}{k}+\binom{n-1}{k-1}\big(\a+2(k-1)(\a-1)\big)\right]|x'|^2+o(|x'|^2).
\end{eqnarray*}
Since $\a>1$, it is direct to see that
\begin{eqnarray}
f(|x'|)>0\hbox{ and }f''(|x'|)<0 \hbox{ near }x'=0.
\end{eqnarray}
Hence $\n^2 f^{-\frac{1}{p_0+k}}+f^{-\frac{1}{p_0+k}} I\geq 0$ is satisfied near the north pole.
As in \cite{GL}, using a lemma in \cite{G}, one may patch a global convex solution to equation (\ref{CMpk}) with some positive function $f$ such that solution is equal to $(1-x_{n+1})^\a$ near the north pole. That is $u=0$ at the north pole and condition \eqref{convexity} is satisfied near the north pole.

\medskip

Next, we will construct a solution to (\ref{CMpk}) for some positive smooth function $f$ satisfying condition (\ref{convexity}) everywhere but $u$ touches $0$. This shows that, a $C^2$ convex hypersuface to the  $k$-th Christoffel-Minkowski problem for $p$-sum with $1< p< k+1$ does not always exist even if $f$ is a smooth positive function such that \eqref{convexity} holds. Hence, the evenness assumption on $f$ cannot be dropped in Theorem \ref{thm1}.

\begin{proposition}\label{thm3}
There exists some $0<\bar p<k$, such that for $0<p_0\leq \bar p$, there is some positive function $f\in C^\infty(\SS^n)$  satisfying \eqref{convexity} and a solution $u$ to \eqref{CMpk} such that $(\n^2 u+u g_{\mathbb{S}^n})\ge 0$ and $u=0$ at some point. Moreover, $u$ is not $C^3$.
 \end{proposition}

\begin{proof} Choose an orthonormal basis $\{e_i\}_{i=1}^n$ on $\SS^n$. For coordinate functions $x_l, l=1, 2, \cdots, n+1$, we know $\n^2_{ij} x_{l}+x_l \delta_{ij}=0$. Since $|\n x_l|^2+x_l^2= |\n x_j|^2+x_j^2$ for any $j\neq l$ and $|\n x|^2+|x|^2= \sum_{i=1}^n |e_i|^2+|x|^2=n+1$, we get $|\n x_l|^2+x_l^2=1$ for any $l=1, 2, \cdots, n+1$.

Let $u(x)=(1-x_{n+1})^\a.$ By direct computations, $$\n_j u=-\a (1-x_{n+1})^{\a-1}\n_j x_{n+1},$$
\begin{eqnarray*}
\n^2_{ij} u&=&\a(\a-1)(1-x_{n+1})^{\a-2}\n_i x_{n+1}\n_j x_{n+1}-\a(1-x_{n+1})^{\a-1}\n^2_{ij} x_{n+1}
\\&=&\a(\a-1)(1-x_{n+1})^{\a-2}\n_i x_{n+1}\n_j x_{n+1}+\a(1-x_{n+1})^{\a-1}x_{n+1}\d_{ij}.
\end{eqnarray*}
Thus
\begin{eqnarray*}
\n^2_{ij} u+u\d_{ij}&=&\a(\a-1)(1-x_{n+1})^{\a-2}\n_i x_{n+1}\n_j x_{n+1}+(1-x_{n+1})^{\a-1}(1+(\a-1)x_{n+1})\delta_{ij}
\\&=&(1-x_{n+1})^{\a-1}(1+(\a-1)x_{n+1})\left[\d_{ij}+\frac{\a(\a-1)\n_i x_{n+1}\n_j x_{n+1}}{(1-x_{n+1})(1+(\a-1)x_{n+1})}\right] .
\end{eqnarray*}
Therefore
\begin{eqnarray*}
\s_k(\n^2_{ij} u+u\d_{ij})&=&(1-x_{n+1})^{k(\a-1)}(1+(\a-1)x_{n+1})^k\times
\\&\times&\left[\binom{n-1}{k}+\binom{n-1}{k-1}\left(1+\frac{\a(\a-1)|\n x_{n+1}|^2}{(1-x_{n+1})(1+(\a-1)x_{n+1})}\right)\right]
\\&=& u^\frac{k(\a-1)}{\a}(1+(\a-1)x_{n+1})^{k-1}\times
\\&\times&\left[\binom{n}{k}(1+(\a-1)x_{n+1})+\binom{n-1}{k-1}\a(\a-1)(1+x_{n+1})\right]
\\&=&\frac{(n-1)!}{k!(n-k)!}u^\frac{k(\a-1)}{\a}(1+(\a-1)x_{n+1})^{k-1}\times
\\&\times&\left[n+k\a(\a-1)+(n+k\a)(\a-1)x_{n+1}\right].
\end{eqnarray*}
In the second equality we used the fact $|\n x_{n+1}|^2=1-x_{n+1}^2$.

Let $f(x)=\frac{(n-1)!}{k!(n-k)!}(1+(\a-1)x_{n+1})^{k-1}\left[n+k\a(\a-1)+(n+k\a)(\a-1)x_{n+1}\right]$ and $\a=\frac{k}{k-p_0}$.
Then $u(x)=(1-x_{n+1})^\a$ with $\a=\frac{k}{k-{p_0}}$ is a solution to $$\s_k(\n^2_{ij} u+u\d_{ij})=u^{p_0} f \hbox{ on }\SS^n. $$

Now let us analyze the function $f$ on $\SS^n$. First, $f$ is smooth. Second it is direct to check that when $\a<2$, i.e. ${p_0}<\frac k 2$, $f>0$.

We claim that when $0<\a-1$ lies in certain range, i.e., ${p_0}\le \bar p$, $f$ satisfies the convexity condition $(\nabla^2 f^{-\frac{1}{k+{p_0}}}+ f^{-\frac{1}{k+{p_0}}} I)>0$.

Let $\tilde g=\tilde f^{-\frac{1}{k+{p_0}}}$, where $$\tilde f=\frac{k!(n-k)!}{(n-1)!}f= (1+(\a-1)x_{n+1})^{k-1}\left[n+k\a(\a-1)+(n+k\a)(\a-1)x_{n+1}\right].$$
We need to show $\left(\frac{\n^2_{ij}\tilde g}{\tilde g}+\d_{ij}\right)>0$.
To simplify the notation, we denote $y=x_{n+1}$.
Direct computations give
$$\frac{\n_i \tilde g}{\tilde g}= -\frac{\a-1}{k+{p_0}}\left[ \frac{k-1}{1+(\a-1)y}+\frac{n+k\a}{n+k\a(\a-1)+(n+k\a)(\a-1) y}\right]\n_i y.$$
\begin{eqnarray*}
\frac{\n^2_{ij} \tilde g}{\tilde g}&=&\frac{\n_i \tilde g \n_j \tilde g}{\tilde g^2}-\frac{\a-1}{k+{p_0}}\left[ \frac{k-1}{1+(\a-1)y}+\frac{n+k\a}{n+k\a(\a-1)+(n+k\a)(\a-1) y}\right]\n^2_{ij} y
\\&&+\frac{\a-1}{k+{p_0}}\left[ \frac{(k-1)(\a-1)}{(1+(\a-1)y)^2}+\frac{(n+k\a)^2(\a-1)}{[n+k\a(\a-1)+(n+k\a)(\a-1) y]^2}\right]\n_{i} y\n_j y.
\end{eqnarray*}
Using $\n^2_{ij} y=-y\d_{ij}$, we have
\begin{eqnarray*}
&&\frac{\n^2_{ij} \tilde g}{\tilde g}+\d_{ij}
\\&=&\left\{\frac{\a-1}{k+{p_0}}\left[ \frac{k-1}{1+(\a-1)y}+\frac{n+k\a}{n+k\a(\a-1)+(n+k\a)(\a-1) y}\right]y+1\right\}\d_{ij}
\\&&+ \Bigg\{\frac{\a-1}{k+{p_0}}\left[ \frac{(k-1)(\a-1)}{(1+(\a-1)y)^2}+\frac{(n+k\a)^2(\a-1)}{[n+k\a(\a-1)+(n+k\a)(\a-1) y]^2}\right]\\&&\quad \quad+\left(\frac{\a-1}{k+{p_0}}\right)^2\left[ \frac{k-1}{1+(\a-1)y}
+\frac{n+k\a}{n+k\a(\a-1)+(n+k\a)(\a-1) y}\right]^2\Bigg\}\n_{i} y\n_j y.
\end{eqnarray*}
Notice that the coefficient of $\n_{i} y\n_j y$ on the RHS of above equation is always positive. To ensure $\left(\frac{\n^2_{ij} \tilde g}{\tilde g}+\d_{ij}\right)$ is positive definite, we only need the coefficient of $\d_{ij}$ on the RHS of above equation is positive, i.e.,
\begin{eqnarray}\label{xxx1}
\frac{\a-1}{k+{p_0}}\left[ \frac{k-1}{1+(\a-1)y}+\frac{n+k\a}{n+k\a(\a-1)+(n+k\a)(\a-1) y}\right]y+1>0.
\end{eqnarray}

Note $k+{p_0}=k+\frac{k(\a-1)}{\a}=k\frac{2\a-1}{\a}$ and the denominator is always positive when $\a<2$. Inequality \eqref{xxx1} is equivalent to say the quadratic form
\begin{eqnarray*}
Q(y)&=&\a(\a-1)\big\{(k-1)[n+k\a(\a-1)+(n+k\a)(\a-1) y]+(n+k\a)[1+(\a-1)y] \big\}y
\\&&+k(2\a-1)[1+(\a-1)y][n+k\a(\a-1)+(n+k\a)(\a-1) y]>0.
\end{eqnarray*}
By  regrouping,
\begin{eqnarray*}
Q(y)&=&k(3\a-1)(\a-1)^2(n+k\a)y^2
\\&&+k(\a-1)\{\a[n+(k-1)\a(\a-1)+\a]+(2\a-1)(2n+k\a^2)\}y
\\&&+k(2\a-1)[n+k\a(\a-1)].
\end{eqnarray*}
By computation, we see that when $0<\a-1$ is close to $0$, i.e., $\bar p$ is sufficiently small,
\begin{eqnarray*}
Q(-1)>0\hbox{ and }Q'(-1)>0\hbox{ and }Q''(y)>0 \hbox{ for }y\in [-1, 1].
\end{eqnarray*}
Therefore when ${p_0}\le \bar p$, we have $Q(y)$ is positive for $y\in [-1, 1]$.

In conclusion, for $0<p_0$  small, we construct a globally defined  function $u$ which is a solution of $\s_k(\n^2_{ij} u+u\d_{ij})=u^{p_0} f$ with a smooth, positive function $f$ with $(\nabla^2 f^{-\frac{1}{k+{p_0}}}+ f^{-\frac{1}{k+{p_0}}} I)>0$. However, $u$ has a zero.\end{proof}

In this example, $(\n^2 u+u g_{\mathbb{S}^n})$ is not of full rank at some point. This implies that for such $f$, the Gauss map fails to be regular and the convex body with support function $u$ is not $C^2$.
However, in the next section we will show that the solution to the PDE \eqref{CMpk} for $\frac{k-1}{2}\le p_0<k$ is always $C^{2}$ when $f$ is $C^2$.

\medskip

\section{$C^2$ estimate for $p_0\ge \frac{k-1}{2}$}

To prove Theorem \ref{thm4'}, we consider the following perturbed equation
\begin{equation}\label{eq_ep} \sigma_k(\n^2 u+(u+\ep) g_{\SS^n})=u^{p_0} f, \end{equation}
for $\ep>0$.

First of all, we prove the following existence for an auxilliary equation below.
\begin{proposition}\label{aux_ex}
For any $v \in C^4(\SS^n)$ with $v>0$ and $f\in C^4(\SS^n)$, there exists a unique  solution $u\in  C^{5,\a}(\SS^n)$ ($0<\a<1$) with $(\n^2 u+vg_{\SS^n})\in \Gamma_k$, which we denote by $T_f(v)$, to
\begin{eqnarray}\label{auxeq}
\s_k(\n^2 u+v g_{\SS^n})=u^{p_0}f.
\end{eqnarray} Moreover, there exists some constant $C$, depending on $n, k, p_0, \a,\|v\|_{C^4}, \|f\|_{C^4}, \min v, \min f$, such that
$$\|u\|_{C^{5,\a}}\le C.$$
\end{proposition}

\begin{proof}
{\bf Step 1.} A priori estimate for \eqref{auxeq}.

Let $u(x_0)=\min u$. Then
$$ \binom{n}{k}v(x_0) \leq \s_k(\n^2 u+v g_{\SS^n})(x_0)= u(x_0)^{p_0}f(x_0).$$
It follows that $u\ge u(x_0)\ge c>0.$
Similarly, we have $u\le C$.

Denote $w_{ij}=u_{ij}+v\d_{ij}.$
Note that $w_{iiss}=w_{ssii}+2w_{ii}-2w_{ss}-v_{ii}+v_{ss}$ for any $i, s$.
For $C^2$ estimate, we can apply the same argument as in the proof of Proposition \ref{estimate} to $\tr(w)= \Delta u+nv$. Once we get the $C^2$ estimate and the positive lower bound of $u$, \eqref{auxeq} is uniformly elliptic. By the Evans-Krylov and the Schauder theory, we have higher order estimate.

{\bf Step 2.} Existence and uniqueness for  \eqref{auxeq}.

To prove the uniqueness, let $u$ and $\tilde{u}$ be two solutions. Then the difference $h= u-\tilde{u}$ satisfies
$a_{ij}(x)h_{ij}+c(x)h=0,$ where $a_{ij}(x)$ is an elliptic operator and $c(x)<0$. Thus $h\equiv 0$ by strong maximum principle.

We use continuity method to prove the existence. We set $f_0:=\frac{1}{v^{p_0}}\s_k(\n^2 v+vg_{\SS^n})$ and $f_t=(1-t)f_0+tf$. Consider \eqref{auxeq}  with $f=f_t$. It is easy to see  $u_0\equiv v$ is the unique solution to \eqref{auxeq} for $f=f_0$.
Next, the kernel of the linearized operator $L_{u_t}$ is trivial and self-adjoint. Thus the openness follows from standard implicit function theorem. The closeness follows from the a priori estimates in Step 1. Therefore, we have the existence of \eqref{auxeq} via
continuity method.
\end{proof}


Next we show the existence for the perturbed equation \eqref{eq_ep}.

 \begin{proposition}\label{est_eqep} Let $\ep>0$ and $\frac{k-1}{2}\le p_0<k$. There exists a  solution $u\in  C^4(\SS^n)$ with $(\n^2 u+(u+\ep)g_{\SS^n})\in \Gamma_k$ to
\eqref{eq_ep}.
Moreover, there exist some positive constants $c_\ep$ and $C_\ep$, depending on $n, k, p_0,  \|f\|_{C^4}, \min f$ and $\ep$, such that
$$u\ge c_\ep\hbox{ and } \|u\|_{C^{5,\a}}\le C_\ep.$$
\end{proposition}
 \begin{proof} {\bf Step 1.} A priori estimate for \eqref{eq_ep}.

 From the equation, $u>0$ automatically. Let $u(x_0)=\min u$, then
 $$\binom{n}{k} \ep^k\le \s_k(\n^2 u+(u+\ep) g_{\SS^n})(x_0)= u(x_0)^{p_0}f(x_0).$$
A positive lower bound $u\ge c_\ep$ follows.
One may follow the same argument in the previous section to prove the $C^1$ and the $C^2$ estimate depending on $c_\ep$. We remark that for these arguments one needs only assume $(\n^2 u+(u+\ep)g_{\SS^n})\in \Gamma_k$, see Remark \ref{rem}.

{\bf Step 2.} Existence for  \eqref{eq_ep}.

We use the degree theory to prove the existence.  Denote by $f_t=(1-t)\binom{n}{k}(1+\ep)^k+tf$ for $t\in [0,1]$.
For any $\o\in C^4(\SS^n)$ and $f_t$ we consider
\begin{eqnarray}\label{auxeq1}
\s_k(\n^2 u+(e^\o+\ep) g_{\SS^n})=u^{p_0}f_t(x).
\end{eqnarray}
From Proposition \ref{aux_ex}, there exists a unique positive solution $T_{f_t}(e^\o+\ep)$ to \eqref{auxeq1}. Define an operator \begin{eqnarray*}
\tilde{T}_t: &C^4&\to C^{5,\a}\\
&\o&\mapsto \log T_{f_t}(e^\o+\ep).
\end{eqnarray*}
It follows from the a priori estimate in Proposition \ref{aux_ex} that $\tilde{T}_t$ is compact.

It is easy to see that $\o$ is a fixed point of $\tilde{T}_t$, i.e., $\o=\tilde T_t(\o)$, if and only if
$u=e^\o$ is a solution to \eqref{eq_ep} with $(\n^2 u+(u+\ep) g_{\SS^n})\in \Gamma_k$.
Therefore, by using the a priori estimates in Proposition \ref{est_eqep}, we see that any fixed point of $\tilde T_t$ is not on the boundary of
$$S_K=\{\o\in C^4: \|\o\|_{C^4}\le K\}$$ when $K$ is sufficient large, depending on $\ep$.

By the degree theory,  $\deg (I-\tilde T_t, S_K, 0)$ is well defined and independent of $t$.

{\bf Claim.} For $t=0$, $u_0\equiv 1$ is the unique solution to \eqref{eq_ep} with $f=f_0$ and the linearized operator $L_{u_0}$ at $u_0\equiv 1$ is injective.

To show this claim, we need the a priori estimate from Proposition \ref{C1 p} and \ref{C2} below, where we  assume $p_0\ge \frac{k-1}{2}$.

First,  there is no other solutions of \eqref{eq_ep} near $u_0\equiv 1$. The linearized operator for the equation \eqref{eq_ep} at $u_0\equiv 1$ is given by
$$L_{u_0}\rho=\s_k^{ij}((1+\ep)g_{\SS^n})(\n^2_{ij}\rho+\rho g_{ij})-p_0f_0\rho=(1+\ep)^{k-1}\binom{n-1}{k-1}(\Delta \rho+n\rho)-p_0\binom{n}{k}(1+\ep)^k\rho.$$ Since the first eigenvalue of $\Delta$ on $\SS^n$ is $n$, we see that the kernel of $L_{u_0}$ is trivial, namely, $L_{u_0}$ is injective. Thus the assertion follows by the implicit function theorem.

Second, for $\epsilon>0$ small,  there exist no other solutions than $u_0\equiv 1$. Suppose there are $\epsilon_l\to 0$ and non-constant solutions $u_l$ for each $\epsilon_l$. By the a priori estimate independent of $\epsilon$ by Propsition \ref{C1 p} and \ref{C2}, there is a subsequence, still denote by $\{u_{l}\}$, with $u_l \to \tilde u$ in $C^{1,\alpha}$, where $\tilde u\in C^{1,1}(\SS^n)$ is a solution of the un-perturbed equation \eqref{CMpk} with $f=f_0$.  It follows from the previous step that $u_l$ is uniformly away from $u_0\equiv 1$, so $\tilde u$ is not the constant $1$, which
   contradicts to the uniqueness of \eqref{CMpk}.

   Third,  for any $\epsilon>0$ such that $u>0$ solves \eqref{eq_ep} with $f=f_0$, the uniqueness is true. This follow immediately from previous two steps. We finish the proof of the claim.

We turn back to the proof of the existence. Since 
$L_{u_0}$ is injective, the derivative $\tilde{T_0}'$ in $C^4$ is injective. The degree can be computed as $\deg (I-T_0, S_K, 0)=(-1)^\beta$ where $\beta$ is the number of eigenvalues of $\tilde{T_0}'$ greater than one. In any case $\deg (I-T_t, S_K, 0)=\deg (I-T_0, S_K, 0)=(-1)^\beta$ is not equal to zero. Therefore we have the existence for \eqref{eq_ep} for any $t\in [0,1]$, in particular for $t=1$. The assertion follows. \end{proof}

We now show the a priori estimate independent of $\epsilon$. The arguments in the proof for $C^1$ estimate in previous section yield the $C^1$ estimate for solutions to \eqref{eq_ep}.

\begin{proposition}\label{C1 p} Let $\ep\ge 0$. Let $u$ be a solution to \eqref{eq_ep} with $(\n^2 u+(u+\ep)g_{\SS^n})\in \Gamma_k$. Then there exists some positive constant $C$, depending on $n, k, p_0, \min f$ and $\|f\|_{C^1}$, but independent of $\ep$, such that $$\| u\|_{C^1}\leq C.$$
\end{proposition}

Next, we show that, in the case  $\frac{k-1}{2}\le p_0<k$, equation \eqref{eq_ep} admits  a $C^2$ estimate independent of $\ep$.

\begin{proposition}\label{C2} Let $\ep\ge 0$. Assume $\frac{k-1}{2}\le p_0<k$. Let $u$ be a solution to \eqref{eq_ep} with $(\n^2 u+(u+\ep)g_{\SS^n})\in \Gamma_k$. Then there exists a nonnegative constant $\a=\a(p_0,k,n)$ depending only on $p_0, k, n$ with
$\a(p_0,k,n)> 0$ if $\frac{k-1}{2}< p_0$, and there is
some positive constant $C$ depending on $n, k, p_0, \min f$ and $\|f\|_{C^2}$, but independent of $\ep$, such that $$\|\frac{\n^2 u}{u^{\a}}\|_{C^0}\leq C.$$
\end{proposition}

\begin{proof} For $k=1$, the standard theory of linear elliptic PDE gives us the  $C^2$ estimate. Hence we consider $k\geq 2$.
In the following proof we denote by $W_u^\ep=\n^2 u+(u+\ep)I$. It is sufficient to prove the upper bound of $\frac{\s_1(W_u^\ep)}{u^{\a}}$ since $W_u^\ep\in \Gamma_2$.

Let $y_0\in \SS^n$ be a maximum point of  $\frac{|\nabla u|^2}{u^{1+\a}}$. Then $$\n |\nabla u|^2(y_0)=(1+\a)|\n u|^2\frac{\n u}{u}(y_0).$$ It follows that $$\frac{(\n^2 u+(u+\ep)I)}{u^{\a}}\cdot\n u(y_0)=(\frac{(1+\a)|\n u|^2}{2u^{1+\a}}+\frac{(u+\ep)}{u^{\a}})\n u(y_0).$$ Thus $\left(\frac{(1+\a)|\n u|^2}{2u^{1+\a}}+\frac{(u+\ep)}{u^{\a}}\right) (y_0)$ is an eigenvalue of $\frac{(\n^2 u+(u+\ep)I)}{u^{\a}}(y_0)$. Since $W_u^\ep\in \Gamma_2$, we have \begin{eqnarray}\label{maxi}
\max \frac{(1+\a)|\n u|^2}{2u^{1+\a}}&=&\frac{(1+\a)|\n u|^2}{2u^{1+\a}}(y_0) \nonumber \\
&\leq & \left(\frac{(1+\a)|\n u|^2}{2u^{1+\a}}+\frac{(u+\ep)}{u^{\a}}\right) (y_0)\nonumber \\ &\leq & \frac{\sigma_1 (W_u^\ep)}{u^{\a}}(y_0)\leq \max  \frac{\sigma_1 (W_u^\ep)}{u^{\a}}.
\end{eqnarray}

Let $x_0$ be a maximum point of $\frac{\sigma_1 (W_u^\ep)}{u^{\a}}$ and by a choice of local frame and a rotation of coordinates we assume $g_{ij}=\d_{ij}$ and $W_u^\ep$ is diagonal at $x_0$. By the maximal condition, at $x_0$,  $$\n \sigma_1=\a\sigma_1\frac{\n u}{u},\quad \frac{\n^2\sigma_1}{\sigma_1}-\frac{\n^2 u^{\a}}{u^{\a}}\leq 0.$$
In the following we compute at $x_0$. Assuming $0\le \a\le 1$, and using Ricci's identity,
\begin{eqnarray}\label{second}
0&\geq &\sigma_k^{ii}[\frac{(\sigma_1)_{ii}}{\sigma_1}-\frac{\a u_{ii}}{u}-\a(\a-1)\frac{u_i^2}{u^2}]\nonumber\\&\ge&\sigma_k^{ii}\frac{(W^\ep_{iiss}+W^\ep_{ss}
-nW^\ep_{ii})}{\sigma_1}-\a k\frac{\sigma_k}{u}+(n+1-k)\a \sigma_{k-1}\nonumber\\&=& \frac{\Delta \sigma_k-\sigma_k^{ij,lm}W^\ep_{ijs}W^\ep_{lms}-nk\s_k}{\sigma_1}-\a k\frac{\sigma_k}{u}+(n+1-k)(1+\a) \sigma_{k-1}.
\end{eqnarray}

From the equation \eqref{eq_ep}, we have
\begin{eqnarray}\label{equ1}
\Delta \s_k&=&p_0 u^{p_0-1}\Delta u f+p_0(p_0-1)u^{p_0-2}|\n u|^2 f\\&&+2p_0u^{p_0-1}\n u\n f+u^{p_0} \Delta f\nonumber
\\&=&p_0u^{p_0-1}\sigma_1 f+p_0(p_0-1)u^{p_0-2}|\n u|^2 f\nonumber\\&&+2p_0u^{p_0-1}\n u\n f+(1-np_0)u^{p_0} \Delta f.\nonumber
\end{eqnarray}
Since $\n \sigma_1=\a\sigma_1\frac{\n u}{u}$,  $\sigma_k=u^{p_0}f$ and  $\nabla \sigma_k=p_0u^{p_0-1}f\nabla u+u^{p_0}\n f$, we deduce from \eqref{concavity} in Lemma \ref{gll-lemma} that
\begin{eqnarray}\label{equ2}
-\sigma_k^{ij,lm}W^\ep_{ijs}W^\ep_{lms}
\geq -\beta u^{p_0-2}|\n u|^2 f+c_1 u^{p_0-1}\n u\n f+c_2u^{p_0} \frac{|\n f|^2}{f},
\end{eqnarray}
where $c_1, c_2$ are constants under control and
\begin{equation}\label{beta}
\beta=(p_0-\a)\frac{(k-2)p_0+k\a}{k-1}.\end{equation}
It follows from \eqref{equ1} and \eqref{equ2} that
\begin{eqnarray}\label{equ3}
&&\Delta \s_k-\sigma_k^{ij,lm}W^\ep_{ijs}W^\ep_{lms}\nonumber\\&\geq &
p_0 u^{p_0-1}\sigma_1 f+\left[p_0(p_0-1)-\beta\right] u^{p_0-2}|\n u|^2 f\nonumber\\&&+c_1u^{p_0-1}\n u\n f-C u^{p_0}
\end{eqnarray}
where $C$ depends on $k, p_0, n, \|f\|_{C^2}$ and $\min f$.

By (\ref{second}) and (\ref{equ3}),
\begin{eqnarray}\label{equ00}  0&\ge& (p_0-k\a) u^{p_0-1} f+ \frac{\left[p_0(p_0-1)-\beta\right] u^{p_0-2}|\n u|^2 f+(c_1+2p_0)u^{p_0-1}\n u\n f-C u^{p_0}}{\sigma_1}\nonumber \\ &&+(n+1-k)(1+\a) \sigma_{k-1}\end{eqnarray}

Note that from \eqref{maxi} we have
\begin{eqnarray}\label{equ44}
\frac{\s_1(W_u^\ep)}{u^{\a}}(x_0)=\max \frac{\s_1(W_u^\ep)}{u^{\a}} \geq \frac{(1+\a)|\nabla u|^2}{2u^{1+\a}}(x_0).
\end{eqnarray}

As $u$ is bounded from above, we deduce from \eqref{equ00} and \eqref{equ44} that
\begin{eqnarray}\label{equ5}
0 &\geq&
\min\left\{(p_0-k\a)+\frac{2(p_0(p_0-1)-\beta)}{1+\a}, p_0-k\a\right\} u^{p_0-1} f\\&&-C \frac{u^{p_0-\frac12}}{\sqrt{\s_1}}-C\frac{u^{p_0}}{\s_1}+(n+1-k)(1+\a) \sigma_{k-1}.\nonumber
\end{eqnarray}
In view of (\ref{beta}), if $p_0\ge \frac{k-1}{2}$, we may choose $\a\ge 0$ such that $p_0-k\a\ge 0$ and \begin{eqnarray*}
&&(p_0-k\a)+\frac{2(p_0(p_0-1)-\beta)}{1+\a}\\&=&\frac{1}{1+\a}\left(\frac{2}{k-1}p_0^2-p_0+\frac{k-5}{k-1}\a p_0-k\a(1+\a)+\frac{2k}{k-1}\a^2\right)\geq 0,
\end{eqnarray*}
 Moreover, if $p_0> \frac{k-1}{2}$, $\a$ can be picked positive.
By the Newton-Maclaurin inequality $$\s_{k-1}\geq C_{n,k}\s_1^{\frac{1}{k-1}}\s_k^{\frac{k-2}{k-1}},$$ it follows from (\ref{equ5}),
\begin{eqnarray}\label{second2}
0 &\ge & (n+1-k)(1+\a)\sigma_1 \sigma_{k-1}-C u^{p_0-\frac12}\sqrt{\s_1}-Cu^{p_0}\nonumber \\
&\geq &C\sigma_1^{1+\frac{1}{k-1}}(u^{p_0}f)^{\frac{k-2}{k-1}}-Cu^{p_0-\frac12}\sqrt{\s_1}-Cu^{p_0}.
\end{eqnarray}
Since $p_0\geq \frac{k-1}{2}$, we can choose $\a\ge 0$ such that  $p_0\frac{k-2}{k-1}\leq p_0-\frac12-(\frac{1}{k-1}+\frac12)\a.$ Moreover, if $p_0> \frac{k-1}{2}$, $\a$ can be picked positive. By virtue of  the uniform upper bound of $u$,  we obtain $\frac{\s_1}{u^\a}\leq C$. The proof is completed.

\end{proof}

\medskip

Now we are ready to prove Theorem \ref{thm4'}.

\noindent{\it Proof of Theorem \ref{thm4'}:} For $\ep>0$, let $u_\ep$ be the solution of
\eqref{eq_ep} with $(\n^2 u_\ep+ (u_\ep+\ep)g_{\SS^n})\in \Gamma_k$. From the a priori $C^2$ estimate independent of $\ep$, there is a subsequence $u_{\ep_i}\to u$ in $C^{1,\a}$ for any $\a<1$. The bound gives $u\in C^{1,1}(\SS^n)$ and $\s_k(\n^2 u+ug_{\SS^n})=u^{p_0}f$ with $(\n^2 u+ u g_{\SS^n})\in \bar \Gamma_k$.

We note that solution $u$  is  $C^2$ continuous if $p_0>\frac{k-1}2$. This follows from Proposition \ref{C2}, $|\nabla^2u(x)|\le Cu^{\a}(x), \forall x\in \mathbb S^n$, because $u\in C^{\infty}$ away from the null set $\{u=0\}$ and $\nabla^2u$ is continuous at every point of $\{u=0\}$ when $\a>0$.\qed

\medskip

We discuss a special case of equation (\ref{CMpk}) when $k=1$. This is the equation corresponding to the $L^p$-Christoffel problem. In this case, equation is semilinear:
\begin{equation}\label{e-c1}
\Delta u(x)+nu(x)=u^{p_0}(x)f(x), \quad x\in \mathbb S^n.\end{equation}
From the $C^1$ estimate established for admissible solutions of equation (\ref{eq_ep})
in the previous section and standard semilinear elliptic theory, we immediately have
\begin{theorem}\label{thmk1}
For any positive function $f\in C^2(\SS^n)$ there exists a nonnegative solution $u$ to \eqref{e-c1}  with
\begin{eqnarray*}
\|u\|_{C^{2,\a}(\SS^n)}\leq C,
\end{eqnarray*}
for some $0<\a<1$ and $C$ depending on $n, k, p_0, \a, \|f\|_{C^2(\SS^n)}$ and $\min_{\SS^n} f$.
 \end{theorem}
If condition (\ref{convexity}) is imposed, one may obtain a sperically convex solution $u$. Though $u\in C^{2,\alpha}$, the corresponding hypersurface with $u$ as its support function may not be in $C^{1,1}$ as $W_u$ may degenerate
on the null set of $u$. Equation (\ref{e-c1}) has variational structure, it is of interest to develop corresponding potential theory as in the classical Christoffel problem \cite{B,F}.

\medskip

To end this paper, we would like to raise the following two questions.

\begin{itemize}
\item[(1)] Using compactness argument as in \cite{GM0}, together with the a priori estimates in Proposition \ref{estimate} and the Constant Rank Theorem \ref{const rank}, one can prove that if $\|f\|_{C^2}+\|\frac1f\|_{C^0}\le M$ and \eqref{convexity} holds, there exists a uniform positive constant $C$ depending only on $n, M$ such that $$W_u\ge C g_{\mathbb{S}^n}.$$
Is there a direct effective estimate of $W_u$ from below under the same convexity conditions, without use of  the constant rank theorem?

\item[(2)] Under the condition of evenness, a positive lower bound of $u$ in Proposition \ref{lowerbound} has been derived via an ODE argument and a bound on $\nabla u$ which depends on $\nabla f$. In the case of $L^p$-Minkowski problem (i.e., $k=n$), one may obtain a bound of volume of the associated convex body $\Omega_u$ from below if $f$ is positive. Is it possible to derive such a priori a positive lower bound of $Vol(\Omega_u)$ for solutions of equation (\ref{CMpk}) in general? This would give a positive lower bound of $u$.
\end{itemize}

\end{document}